\documentclass[12pt]{amsart}

\usepackage{amsfonts,amsmath,latexsym,amssymb,verbatim,amsbsy}
\usepackage{amsthm}

\usepackage{pstricks}

\theoremstyle{plain}
\newtheorem{THEOREM}{Theorem}[section]

\newtheorem{theorem}[THEOREM]{Theorem}
\newtheorem{corollary}[THEOREM]{Corollary}
\newtheorem{lemma}[THEOREM]{Lemma}

\theoremstyle{definition}

\theoremstyle{remark}

\newcommand{\thm}[1]{Theorem~\ref{#1}}
\newcommand{\lem}[1]{Lemma~\ref{#1}}

\newcommand{\bb}{\begin{equation}}
\newcommand{\ee}{\end{equation}}
\newcommand{\bq}{\begin{eqnarray}}
\newcommand{\eq}{\end{eqnarray}}
\newcommand{\bqn}{\begin{eqnarray*}}
\newcommand{\eqn}{\end{eqnarray*}}

\newcommand{\N}{\ensuremath{\mathbb{N}}}   
\newcommand{\R}{\ensuremath{\mathbb{R}}}   

\def \a {\alpha}
\def \b {\beta}
\def \d {\delta}

\def \e {\varepsilon}

\def \n {\nabla}
\def \s {\sigma}
\def \r {\rho}

\def \D {\Delta}

\def \t {\tau}
\def \w {\omega}
\def \O {\Omega}

\def \H {\mathcal{H}}
\def \E {\mathcal{E}}

\def \ra {\rightarrow}
\def \ss {\subset}

\DeclareMathOperator{\dist}{dist} %

\begin{document}

\title[Energy concentration and drain]{A study of energy concentration and drain  in incompressible fluids}
\author{Roman Shvydkoy}
\thanks{This work was partially supported
by NSF grant DMS--0907812.}
\address[R. Shvydkoy]{University of Illinois at Chicago, Department of Mathematics
(M/C 249), Chicago, IL 60607, USA} %
\email{shvydkoy@math.uic.edu}


\begin{abstract} In this paper we examine two opposite scenarios of energy behavior for solutions of the Euler equation. We show that if $u$ is a regular solution on a time interval $[0,T)$ and if $u \in L^rL^\infty$ for some $r\geq \frac{2}{N}+1$, where $N$ is the dimension of the fluid, then the energy at the time $T$ cannot concentrate on a set of Hausdorff dimension smaller than $N - \frac{2}{r-1}$. The same holds for solutions of the three-dimensional Navier-Stokes equation in the range $5/3<r<7/4$. Oppositely, if the energy vanishes on a subregion of a fluid domain, it must vanish faster than $(T-t)^{1-\d}$, for any $\d>0$. The results are applied to find new exclusions of locally self-similar blow-up in cases not covered previously in the literature.
\end{abstract}

\maketitle

\section{Introduction}

We consider evolution of an incompressible $N$-dimensional ideal fluid governed by the Euler equations
\begin{equation}\label{ee}
\begin{split}
u_t + u \cdot \n u + \n p & =0\\
\n \cdot u & = 0,\\
u(t=0)& = u_0.
\end{split}
\end{equation}
Here $u$ is the velocity field, $p$ internal pressure, and we assume the fluid domain is $\R^N$. For any initial condition $u_0 \in H^{N/2 +1+ \e}$ there exists a unique local-in-time  solution $u \in C([0,T);  H^{N/2+1 + \e})$ with the associated pressure given by
\begin{equation}\label{p}
p(t,x) = - \frac{|u(t,x)|^2}{N} + P.V. \int_{\R^N} K_{ij}(x-y) u_i(y)u_j(y)
\, dy,
\end{equation}
where $K_{ij}(y) = \frac{N y_i y_j - \d_{ij}|y|^2}{N \w_N |y|^{N+2}}$, and $\w_N = 2\pi^{N/2}(N\Gamma(N/2))^{-1}$ is the volume of the unit ball in $\R^N$. A recent trend in the global regularity problem for \eqref{ee} is to rule out model scenarios of blow-up that arise in numerical simulations.  The model of particular relevance to this note is the locally self-similar blow-up given by 
\begin{equation}\label{e:ss}
\begin{split}
u(x,t)  &= \frac{1}{(T - t)^{\frac{\a}{1+\a}}}\, v\left( \frac{x-x_0}{(T - t)^{\frac{1}{1+\a}}}\right),\\
 p(x,t) & =  \frac{1}{(T - t)^{\frac{2\a}{1+\a}}}\, q\left( \frac{x-x_0}{(T - t)^{\frac{1}{1+\a}}}\right),
\end{split}
\end{equation}
for $|x-x_0|<\rho_0$, $t<T$, and $\a>-1$ (focusing case). These solutions emerge, for instance, in vortex line models of Kida's high-symmetry flows (see Pelz and others \cite{bp,kimura,ngb,pelz}), although previously self-similar blow-up has been observed as well, \cite{kerr,bmvps}. In a recent joint effort with D. Chae \cite{cs} (see also \cite{chae-11,he-ext,schonbek}) solutions of the form \eqref{e:ss} have been ruled out under additional integrability condition, $v \in L^p(\R^N) \cap C^1_{loc}(\R^N)$, $p\geq 3$, in the range $-1 < \a \leq \frac{N}{p}$ and $\a > \frac{N}{2}$. In the energy conservative scaling $\a = \frac{N}{2}$, self-similar solutions are excluded provided $v\in L^2$ and the power bounds
$
\frac{1}{|y|^{N+1-\d}} \lesssim |v(y)| \lesssim |y|^{1-\d}
$
hold at infinity. The range $\frac{N}{p} < \a < \frac{N}{2}$, or just $\a < \frac{N}{2}$ if no $L^p$-condition is assumed, remains open at the moment. We remark that finiteness of the total energy of $u$ requires $v$ to satisfy the energy growth bound
$
\int_{|y| \leq L } |v(y)|^2\, dy \lesssim L^{N-2\a}$, instead of $v \in L^2$.
It follows from the arguments of \cite{cs}, that a slightly better bound
\begin{equation}\label{enweak}
\int_{|y| \leq L } |v(y)|^2\, dy \lesssim L^{N-2\a}o(1),
\end{equation}
along with $v\in L^p$, $p \geq 3$, implies $v = 0$. 

The main motivation of this present work is to understand the general energetics of the Euler system that lies behind the results of \cite{cs}, and consequently to exclude new cases of self-similar blow-up in the range $\a \leq \frac{N}{2}$.  To illustrate the thrust of what follows let us consider a self-similar solution with $\a = \frac{N}{2}$. As $t \ra T$, the energy density $|u(x,t)|^2$ tends weakly to the Dirac mass at $x_0$. We see two different types of anomaly: energy concentration on the "small" set $\{x_0\}$, and energy drain elsewhere. One can study these phenomena for general solutions of \eqref{ee} by introducing an energy measure at time $T$. Suppose that $u$ is a smooth solution of \eqref{ee} on the interval $[0,T)$. Then the local energy equality
\begin{equation}\label{locen}
\begin{split}
\int_{\O} |u(t,x)|^2 \s(x) \, dx = \int_{\O} |u_0(x)|^2 \s(x) \, dx  \\
+ \int_{0}^{t}  \int_{\O} ( |u|^2 + 2p) u \cdot \n \s \, dx\, d\t,
\end{split}
\end{equation}
holds for all $0<t<T$, and $\s \in C^\infty_0(\R^N)$. If, in addition, $u \in L^3([0,T); L^3(\O))$ on some subdomain $\O \ss \R^N$, then \eqref{locen} guarantees that the limit of the right hand side exists as $t \ra T$ for all $\s \in C^\infty_0( \O)$, and hence, so does the limit on the left hand side. It therefore defines a non-negative measure on $\O$, which we call the energy measure and denote by $\E_T$. If the solution does not loose smoothness at time $T$, then trivially, 
\begin{equation}\label{energyreg}
d\E_T (x) = |u(T,x)|^2 dx.
\end{equation}
Thus, deviation from \eqref{energyreg} can be viewed as a measure of severity of the blow-up. One way to quantify this deviation is to consider how low the Hausdorff dimension of a set of positive $d\E_T$-measure can be: 
\begin{equation}\label{dim}
d_T = \inf \{ d\geq 0: \exists S \subset \O, \dim_{\H}(S) \leq d, \E_T(S)>0 \}.
\end{equation}
In \thm{t:r} we will prove that the size of $d_T$ can be controlled from below by the growth of $L^\infty$-norm at the time of blow-up: if $u\in L^rL^\infty(\O)$, for  some $r\geq \frac{2}{N}+1$, then $d_T \geq N - \frac{2}{r-1}$; and $\E_T$ has no atoms if $r = \frac{2}{N}+1$ (a more general statement is given in \lem{l:atom}). In the case of a self-similar solution, this translates into the following statement: if $v \in L^\infty$ in \eqref{e:ss}, then the energy at time $T$ cannot concentrate on sets of dimension smaller than $N-2\a$. Unfortunately, our technique does not rule out concentration to a point under milder condition $\|u(\cdot)\|_\infty \in  L^{\frac{2}{N}+1}_{weak}$ which is the kind of condition that appears in the case $\a = \frac{N}{2}$. Instead, we will examine this case, as well as $\a \leq N/2$, from the point of view of the opposite phenomenon -- the energy drain. 

In \thm{t:merger} we interpret energy drain as a merger of two solutions, given and the trivial one.  As a consequence of \thm{t:merger}, if $u \in L^r L^\infty$, for some $r>1$, and $\|u(t)\|_2 \ra 0 $ on a domain $\O$ at time $T$, then the following improvement occurs $\|u(t)\|_2 \leq C_\d (T-t)^{1-\d}$, for all $\d>0$. This can be interpreted as curbing the influence of pressure on local uniqueness. Let us now consider self-similar blow-up in the case $\a = N/2$. Condition $v\in L^2$ 
ensures drain of energy  in the annulus $\r_0/2 < |x-x_0| < \r_0$. Hence, the improved rate in self-similar variables translates into the bound 
$\int_{L < |y| < 2L} |v|^2\, dy \lesssim L^{-N-2 - \d}$. This rate of decay of energy is strong enough to put $v$ automatically in all $L^p$ spaces for $\frac{N}{N+1} < p <N+4+\d$. Obviously, the lower bound $|v| \geq |y|^{-N-1+\d}$, even in a sector, is inconsistent with these implications. We thus obtain a more robust exclusion condition. In the case $\a <N/2$,  condition \eqref{enweak} ensures energy drain in the entire region of self-similarity $|x-x_0| < \r_0$. The improved rate gives the bound  $\int_{|y| \leq L} |v|^2\, dy \lesssim L^{N-2 -4\a + \d}$, for all $\d>0$. This implies $v=0$ in the range $\frac{N-2}{4} < \a < \frac{N}{2}$, without any $L^p$-condition as previously considered in \cite{cs}. The full list of exclusions based on energy drain is stated in Corollary~\ref{c:ss}.

Adaptations of the general results above to the Navier-Stokes system is given in Section~\ref{s:NSE}. We show that the estimates on the linear term in many cases are subordinate to those obtained for the non-linear term. However we believe that a more meaningful use of the parabolic nature of the equation may improve the results.

\section{Energy concentration}

\begin{theorem}\label{t:r}
Suppose $u \in L^r([0,T); L^\infty(\O))$ for some $r\geq \frac{2}{N}+1$. Then $d_T \geq N - \frac{2}{r-1}$. Moreover, if $r = \frac{2}{N}+1$, then $\E_T$ has no atoms.
\end{theorem}

\subsection{Case $r = \frac{2}{N}+1$} The case $r = \frac{2}{N}+1$ is in fact straightforward. Suppose $ \E_T(\{ x_0\}) = \e_0 >0$. Let us fix a small $\r >0$ so that $B_{4\r}(x_0) \ss \O$.
Let us fix a $C^\infty$-function $0 \leq \s \leq 1$ with $\s(x) = 1$ on $B_{1/2}(0)$ and $\s(x) = 0$ on $\R^N \backslash B_1(0)$. Let $\s_\r(x) = \s((x-x_0)/\r)$. From the energy equality  \eqref{locen} we obtain for all $t < T$
\[
\begin{split}
\e_0 &\leq \int_\O |u(t,x)|^2 \s_\r(x)\, dx + \int_t^T \int_\O (|u|^2 + 2p)u \cdot \n \s_\r\, dx d\t \\
&\leq C \int_{|x-x_0|\leq \rho} |u(t,x)|^2 \, dx + \frac{C}{\r}\int_t^T \int_{|x-x_0| \leq \r} (|u|^3 + |u||p|)\, dxd\t.
\end{split}
\]
Note that $L^{\frac{2}{N}+1}L^\infty \cap L^\infty L^2 \ss L^3L^{\frac{3N}{N-1}}$. Thus,
\[
\frac{1}{\r}\int_t^T \int_{|x-x_0| \leq \r} (|u|^3 + |u||p|)\, dxd\t \leq \|u\|^3_{L^3([t,T);L^{\frac{3N}{N-1}})}.
\]
So, letting $\r \ra 0$ first and then $t \ra T$, we arrive at the contradictory statement $\e_0 = 0$.

Further analysis reveals that if energy concentration is to occur it has to happen over a family of shrinking balls. 
\begin{lemma}\label{l:atom} Suppose $u \in L^1([0,T); L^\infty(\O))$. Suppose that for any $A>0$,
\begin{equation}\label{enlimit}
\liminf_{t\ra T}\int_{|x-x_0| \leq A \int_t^T \|u(\t)\|_\infty d\t} |u(t,x)|^2\, dx = 0.
\end{equation}
Then $\E_T$ has no atom at $\{x_0\}$.
\end{lemma}

One can easily check that $u \in L^{\frac{2}{N}+1}L^\infty$, or a weak-type condition $\|u(t)\|_\infty \leq o(T-t)(T-t)^{\frac{N}{N+2}}$ implies \eqref{enlimit}.

\begin{proof}
Let $E = \frac{1}{2}\int_{\R^3} |u|^2 \,dx$ denote the total (time-independent) energy of the solution. 
The cubic term in the last integral has a straightforward estimate
\[
 \int_{|x-x_0| \leq \r} |u|^3\, dx \leq E \|u(\t)\|_\infty.
\]
We split the pressure as follows $p = p_{loc}+p_1 + p_2$, where $p_{loc}$ is the local term of \eqref{p}, and does not require attention, while
\[
p_1(x) = \int_{|y - x_0| < 2\r}; \qquad p_2(x) = \int_{|y - x_0| \geq 2\r}.
\]
We have, by the Calderon-Zygmund boundedness,
\begin{equation*}\label{}
\begin{split}
\int_{|x-x_0| \leq \r} |u||p_1| \, dx \leq \left(  \int_{|x-x_0| \leq \r} |u|^2\, dx \right)^{1/2} \left(  \int_{|x-x_0| \leq 2\r} |u|^4\, dx \right)^{1/2} \\ \leq E \|u(\t) \|_\infty.
\end{split}
\end{equation*}
as to the $p_2$ term we have
\begin{equation*}\label{}
\begin{split}
\int_{|x-x_0| \leq \r} |u||p_2| \, dx \leq  C  \|u(\t) \|_\infty  \r^N \sup_{|x-x_0| \leq \r}\int _{|y - x_0| \geq 2\r} \frac{1}{|x-y|^N} |u(y)|^2\, dy \\
\leq C E \| u(\t)\|_\infty.
\end{split}
\end{equation*}
Returning to the energy inequality, we obtain
\begin{equation}\label{}
\begin{split}
\e_0 \leq C  \int_{|x-x_0|\leq \rho} |u(t,x)|^2 \, dx+ \frac{CE}{\r} \int_t^T  \|u(\t)\|_\infty \, d\t
\end{split}
\end{equation}
For any fixed $A>0$, let $\r =  A \int_t^T  \|u(\t)\|_\infty \, d\t$. Letting $t \ra T$, the above implies $\e_0 \lesssim A^{-1}$, a contradiction.
\end{proof}

\subsection{Case $r > \frac{2}{N}+1$}
Let us fix an $0< L<1$, the subdomain $\O_L = \{x \in \O: \dist\{x, \d\O\} >L\}$, and without loss of generality assume that $S \ss \O_L$. Let us fix a small $\d>0$ and large $M \in \N$ such that 
\begin{equation}\label{M}
M \d > N - \frac{2}{r-1}.
\end{equation}
The main technical ingredient is the following lemma.

\begin{lemma}\label{l:par}
There exists a constant $C= C(\d,L,u) >0$ and $\r_0 = \r_0(\d,L,u)>0$ such that for every $x_0 \in \O_L$ and $\r < \r_0$ one has
\begin{equation}\label{e:par}
\sup_{T - \r^{\frac{r}{r-1} + \d}\leq t\leq T} \int_{|x-x_0| \leq \r} |u(t,x)|^2\, dx \leq C \r^{ N - \frac{2}{r-1} -\d }.
\end{equation}
\end{lemma}
The theorem follows immediately from \lem{l:par}. Indeed, suppose that $\dim_\H(S) =d<N - \frac{2}{r-1}$. Let $\d>0$ be so small that $d < N - \frac{2}{r-1} - \d$. Then for every $\e>0$ there exists a cover of $S$ by open balls $B_{\r_i}(x_i)$, $\r_i < \r_0$, $x_i \in \O_L$, with 
\[
\sum_{i = 1}^\infty \r_i^{ N - \frac{2}{r-1} - \d} <\e.
\]
Then
\begin{equation*}\label{}
\begin{split}
\E_T(S) &\leq \sum_{i = 1}^\infty \int_\O \s_{2\r_i}(x-x_i)d \E_T(x) =  \sum_{i = 1}^\infty \lim_{t \ra T} \int_\O \s_{2\r_i}(x-x_i) |u(t,x)|^2 \, dx\\
& \leq C \sum_{i = 1}^\infty \r_i^{ N - \frac{2}{r-1} - \d} \leq C \e.
\end{split}
\end{equation*}

\begin{proof}[ Proof of \lem{l:par}]
To simplify our notation, let us assume that $x_0 = 0$, $T = 0$, and the time $t>0$ is reversed. Let us introduce some notation first. Let us fix a positive time $t_0>0$, and $\r<\r_0$ where $\r_0$ is small, but fixed, satisfying
\begin{equation}\label{}
\log_2(L/\r_0 )> 3M+4.
\end{equation}
The time $t_0$ will be determined later and will depend on $\r$. Let $\s \in C^\infty_0(\R^N)$ with $\s = 1$ on $|x|<1/2$ and $\s = 0$ on $|x|>1$. Denote
\[
\begin{split}
E(t,\r) & = \int |u(x,t)|^2 \s(x/\r)\, dx \\
f(t) & = \|u(t)\|_{L^\infty(\O)} \\
F(t)& = \int_t^{t_0} f(\t) d\t.
\end{split}
\]
Note that $\int_{|x| \leq \r} |u(x,t)|^2\, dx \leq E(t,2\r)$. Our goal will be to establish uniform bounds on the energy $E(t,\r)$ for all $t<t_0$ and $\r <\r_0$. In all computations below $\lesssim$ will denote an inequality that holds up to a constant independent of $\r$ or $t$.

Let $\r <\r_0$ be fixed, and let $K  = [\log_2(L/\r)] - 1$. From the above, $K-3M \geq 3$. Notice that the ball $\{|x| < 2^K \r\}$ is still inside the domain $\O$. From the energy equality \eqref{locen} we find
\begin{equation}\label{enini}
E(t,\r) \leq E(t_0,\r) + \frac{C}{\r} \int_t^{t_0} \int_{|x| \leq \r}(|u|^3 + |u||p|)\, dx\, d\t.
\end{equation}
Using introduced notation, we have
\[
\int_t^{t_0} \int_{|x| \leq \r}|u|^3 \, dx\, d\t \leq \int_t^{t_0} f(\t) E(\t, 2\r) \, d\t.
\]
As to the pressure term we split similar to the previous 
$p = p_{loc}+p_1 + p_2$, where
\[
p_1(x) = \int_{|y| < 2\r}; \qquad p_2(x) = \int_{|y| \geq 2\r}.
\]
The term with $p_1$ is estimated through the Calderon-Zygmund inequality as before, and gives
\[
 \int_t^{t_0} \int_{|x| \leq \r} |u||p_1| \, dx\, d\t \lesssim \int_t^{t_0} f(\t) E(\t, 4\r) \, d\t.
\]
For $p_2$ we obtain for all $|x| \leq \r$,
\[
p_2(\t,x) \leq \frac{1}{\r^N} \sum_{k=3}^{K - 3M} \frac{1}{2^{Nk}} E(\t, 2^k \r) + \frac{2^{3NM}}{L^N} \|u\|_2.
\]
Thus,
\[
\int_t^{t_0} \int_{|x| \leq \r} |u||p_2| \, dx\, d\t \lesssim \int_t^{t_0}f(\t) \sum_{k=3}^{K - 3M}  \frac{1}{2^{Nk}}E(\t, 2^k \r)\, d\t + \r^N F(t).
\]
Putting all the estimates together we obtain
\begin{equation}\label{ini}
\begin{split}
E(t,\r) \lesssim E(t_0,\r) + \frac{1}{\r}   \sum_{k=1}^{K - 3M}  \frac{1}{2^{N k}}  \int_t^{t_0}f(\t) E(\t, 2^k \r) \, d\t + \r^{N-1} F(t).
\end{split}
\end{equation}
Notice that the index $k$ must be allowed to reach $3$ for the above step to be possible. With our choice of $K$, we therefore can make $M$ iterations of \eqref{ini}. Let $k_1 = k$ and $t_1 = \t$. From \eqref{ini} we obtain
\[
\begin{split}
E(t_1, 2^{k_1} \r) \lesssim E(t_0,2^{k_1} \r)&  + \frac{1}{2^{k_1} \r}   \sum_{k_1+k_2 \leq K - 3(M-1)}  \frac{1}{2^{N k_2}}  \int_{t_1}^{t_0}f(t_2) E(t_2, 2^{k_1+k_2} \r) \, dt_2\\
& + 2^{(N-1)k_1} \r^2 F(t_1).
\end{split}
\]
Plugging back into \eqref{ini} we obtain
\[
\begin{split}
E(t,\r) & \lesssim E(t_0,\r) + \frac{1}{\r} F(t) \sum_{k_1=1}^{K - 3M}  \frac{1}{2^{N k_1}} E(t_0, 2^{k_1} \r) \\
&+ \r^{N-2} \sum_{k_1=1}^{K - 3M}  \frac{1}{2^{k_1}}  \int_{t}^{t_0}f(t_1) \int_{t_1}^{t_0}f(t_2)\, dt_2 dt_1 \\
&+ \frac{1}{\r^2} \sum_{k_1+k_2 \leq K - 3(M-1)}  \frac{1}{2^{N(k_1+k_2)}2^{k_1}}  \int_{t}^{t_0}f(t_1) \int_{t_1}^{t_0}f(t_2) E(t_2, 2^{k_1+k_2} \r) \, dt_2 dt_1.
\end{split}
\]
Estimating all the energies at time $t_0$ trivially $E(t_0, l) \leq l^N f^2(t_0)$ we obtain
\[
\begin{split}
E(t,\r) & \lesssim \r^N f^2(t_0) + K \r^{N-1} f^2(t_0) F(t) + \frac{1}{2!} F^2(t) \r \\
&+ \frac{1}{\r^2} \sum_{k_1+k_2 \leq K - 3(M-1)}  \frac{1}{2^{N(k_1+k_2)}2^{k_1}}  \int_{t}^{t_0}f(t_1) \int_{t_1}^{t_0}f(t_2) E(t_2, 2^{k_1+k_2} \r) \, dt_2 dt_1.
\end{split}
\]
The next iteration produces the following inequality
\[
\begin{split}
E(t,\r) & \lesssim \r^N f^2(t_0)K \left(1 + F(t)/\r + \frac{1}{2!} (F(t)/\r)^2 \right)  \\
&+ \r^3\left(\frac{1}{2!} (F(t)/\r)^2 + \frac{1}{3!} (F(t)/\r)^3 \right) \\
&+ \frac{1}{\r^3} \sum_{k_1+k_2 +k_3 \leq K - 3(M-2)}  \frac{1}{2^{N(k_1+k_2)}2^{k_1+k_2}}  \\
&\int_{t}^{t_0}f(t_1) \int_{t_1}^{t_0}f(t_2) \int_{t_2}^{t_3} f(t_3) E(t_3, 2^{k_1+k_2+k_3} \r) \, dt_3 dt_2 dt_1.
\end{split}
\]
On the $M$-th step we obtain
\[
\begin{split}
E(t,\r) & \lesssim K \r^N f^2(t_0)\left( 1 + \ldots  + \frac{1}{(M-1)!} (F(t)/\r)^{(M -1)} \right) \\
& + \r^N \left( \frac{1}{2!} (F(t) /\r)^2 + \ldots + \frac{1}{M!} (F(t) / \r)^M \right)\\
&+ \frac{1}{\r^M} \mathrm{sum}_M
\end{split}
\]
where the "sum$_M$" involves an $M$-tuple integral similar to the above. Replacing the energies in that integral trivially by $\|u\|_2$, we obtain
\[
\frac{1}{\r^M} \mathrm{sum}_M \lesssim (F(t)/\r)^M.
\]
We thus arrive at the following estimate for all $0<t<t_0$
\[
E(t,\r)  \lesssim K \r^N f^2(t_0)\exp{(F(t)/\r)} + \r^N \exp{(F(t)/\r)}+ (F(t) / \r)^M. 
\]
By H\"{o}lder, $F(t) \leq t_0^{\frac{r-1}{r}} \|u\|_{L^rL^\infty(\O)}$. Suppose we can choose $t_0 \sim \r^{\frac{r}{r-1} + \d} $ such that $f^r(t_0) \lesssim \frac{1}{t_0}$. Then, in view of \eqref{M}, and $K \lesssim | \log_2\r | \lesssim \r^\d$,  the bound above would give the desired inequality \eqref{e:par}. To find such $t_0$ we recall that $f \in L^r(0,T)$. Then starting from some $t' >0$ for all $t<t'$ there exists a $t_0 \in [t/2, t] $ such that $f^r(t_0) \leq 1/t_0$. Indeed, otherwise the integral of $f$ would diverge logarithmically. In addition, $t'$ depends only on $\|f\|_r$. Therefore by further reducing the size of $\r_0$ to satisfy $ \r_0^{\frac{r}{r-1} + \d} \leq t'$ we obtain the desired conclusion.
\end{proof}

Applying \thm{t:r} to the case of self-similar solutions \eqref{e:ss} with $v\in L^\infty$ we immediately obtain the following conclusion.

\begin{corollary} Suppose $v \in L^\infty$, $0< \a <\frac{N}{2}$. Then in the region of self-similarity the energy does not concentrate on sets of dimension smaller than $N-2\a$.
\end{corollary}

\section{Local merger and self-similar solutions}\label{s:merger}

\begin{theorem}\label{t:merger}
Suppose $u_1$ and $u_2$ are two classical solutions to \eqref{ee} on a time interval $[0,T)$, and $\n u_1, u_2 \in L^r([0,T);L^\infty(\O))$ or $\n u_2, u_1 \in L^r([0,T);L^\infty(\O))$ for some $r >1$.  Suppose further that 
\begin{equation}\label{e:merger}
\| u_1 (t)- u_2(t) \|_{L^2(\O)} \ra 0 \text{, as } t \ra T.
\end{equation}
Then for every compactly embedded subdomain $\O' \ss \O$ and every $\d >0$ there exists $C>0$ such that
\begin{equation}\label{e:improvemerger}
\| u_1 (t)- u_2(t) \|_{L^2(\O')} \leq C (T - t)^{1-\d}, \text{ as } t \ra T.
\end{equation}
\end{theorem}
\begin{proof}
Let us assume for definiteness that $\n u_1, u_2 \in L^r([0,T);L^\infty(\O))$. Let $w = u_1-u_2$. Then $w$ satisfies
\begin{equation}\label{ }
w_t = -w \cdot \n u_1 - u_2 \cdot \n w - \n q,
\end{equation}
where $q$ is the associated pressure recovered via a relationship similar to \eqref{p}.
By the standard covering argument it suffices to show that for every $x_0 \in \O$ there exists a $\r_0>0$ such that 
\begin{equation}\label{goal10}
 \int_{|x-x_0| \leq \r_0} |w(t,x)|^2\, dx \leq C(T-t)^{2-\d}.
\end{equation}
Let us assume for notational convenience that $x_0 = 0 \in \O$. Let us denote
\[
E(t,\r)  = \int_{|x| \leq \r} |w(t,x)|^2\, dx.
\]
The result will follow from the following energy estimate
\begin{equation}\label{vanen3}
E(t,\r)  \lesssim \int_t^T f(\t) E(\t, 4\r) \, d\t + \int_t^T\sqrt{ E(\t, 4\r)}\, d\t,
\end{equation}
where $f = \|u_1\|_\infty + \|\n u_1\|_\infty + \|u_2\|_\infty $. To see that, first, let us note the following interpolation inequality
\[
\|u_1\|_\infty \leq C_\O \|u_1\|_2^{\frac{2}{N+2}}( \|u_1\|_\infty + \|\n u_1\|_\infty)^{\frac{N}{N+2}}.
\]
Since the energy of $u_1$ is bounded, we obtain $\| u_1\|_\infty^{\frac{N+2}{N}} \leq \max\{ C_{\O,E} ; \|\n u_1\|_\infty \}$, and hence $f \in L^r$. Let us now consider
\[
\a_0 = \sup\{\a <2: \exists \r > 0, \exists C>0: E(t,\r) \leq C(T-t)^\a \text{ as } t\ra T\}.
\]
From \eqref{vanen3} it follows that $\a_0 \geq 1- 1/r$, and \eqref{goal10} is equivalent to $\a_0 \geq 2$. Assume that $\a_0 < 2$. Then let $\d >0$ be small, and $\a > \a_0 - \d$. Let $\r>0$ be as in the definition. Then, from \eqref{vanen3},
\begin{equation}\label{}
\begin{split}
E(t,\r/4)  & \lesssim  \int_t^T f(\t) E(\t, \r) \, d\t + \int_t^T\sqrt{ E(\t, \r)}\, d\t \\
& \lesssim (T-t)^{1- \frac{1}{r} + \a_0 - \d} + (T-t)^{\frac{\a_0 - \d + 2}{2}}.
\end{split}
\end{equation}
The power now is strictly grater than $\a_0$ provided $\d$ is small enough, which is a contradiction.

We now turn to proving \eqref{vanen3}. From the energy equality, for all $t<T-\e<T$,
\begin{equation}\label{en-back}
\begin{split}
E(t,\r) & \leq E(T-\e,\r) + \frac{1}{\r} \int_t^{T-\e} \int_{|x| \leq 2\r} ( |u_2 | |w|^2 + |w||q|) \, dx d\t \\
&+  \int_t^{T-\e} \int_{|x| \leq 2\r}  |\n u_1 | |w|^2  \, dx d\t.
\end{split}
\end{equation}
Letting $\e \ra 0$ we obtain
\begin{equation}\label{envan}
\begin{split}
E(t,\r) & \leq C_\r \int_t^T \int_{|x| \leq 2\r} ( |u_2 | |w|^2 + |w||q| + |\n u_1 | |w|^2 ) \, dx d\t.
\end{split}
\end{equation}
We have
\[
\int_{|x| \leq 2\r}  (|u_2 | + |\n u_1|) |w|^2\, dx \leq (\|u_2(\t) \|_\infty + \|\n u_1(\t) \|_\infty) E(\t, 2\r).
\]
The local part of the pressure adds another term $ \| u_1(\t) \|_\infty E(\t, 2\r)$. The non-local part is split as before,
\[
\begin{split}
q_1(x) &= \int_{|y|\leq 4\r} K_{ij}(x-y) (w_i u^{(1)}_j + w_i u^{(2)}_j)(y) dy \\
q_2(x) &= \int_{|y| > 4\r} K_{ij}(x-y) (w_i u^{(1)}_j + w_i u^{(2)}_j)(y) dy.
\end{split}
\]
We have 
\[
\begin{split}
\int_{|x| \leq 2\r} |w||q_1| \, dx & \leq \left( \int_{|x| \leq 2\r} |w|^2 dx \right)^{1/2} \left( \int_{|x| \leq 4\r} |w|^2 (|u_1| + |u_2|)^2 dx \right)^{1/2} \\
& \leq E(\t,4\r) (\|u_1\|_\infty + \|u_2\|_\infty),
\end{split}
\]
while simply using that  $|q_2| \leq C_\r \|u_1\|_2 \|u_2\|_2$ we have
\[
\int_{|x| \leq 2\r} |w||q_2| \, dx \leq C_\r  \sqrt{E(\t, 2\r)}.
\]
Incorporating the above estimates into \eqref{envan} we obtain
\begin{equation*}\label{vanen2}
E(t,\r)  \lesssim \int_t^T ( \|u_1\|_\infty + \|\n u_1\|_\infty + \|u_2\|_\infty ) E(\t, 4\r) \, d\t + \int_t^T\sqrt{ E(\t, 4\r)}\, d\t,
\end{equation*}
which is the desired inequality.
\end{proof}

Assuming $u_2 = 0$ in the above theorem we can find several new exclusion criteria for self-similar blow-up. First let us assume that $v\in C^1_{loc}$ and $|v(y)| \lesssim |y|^{1-\d}$, as $y \ra \infty$, for some $0< \d<1$, and $\a >0$. This puts the solution $u$ in $L^rL^\infty$ for some $r>1$ in the region of self-similarity. In the case $\a = N/2$, the natural energy assumption $v \in L^2$ implies that $\|u(t)\|_2 \ra 0$ in any annulus $0<\r_1< |x - x_0| < \r_0$. Thus, the solution merges with the trivial zero solution, and hence, \thm{t:merger} implies
\[
\int_{L < |y| < 2L} |v|^2\, dy \lesssim \frac{1}{L^{N+2 - \d'}},
\]
for all $\d'>0$. Now, by H\"older for any $p<2$ we have
\[
\int_{L < |y| < 2 L} |v(y)|^p dy \lesssim L^{N - \frac{Np}{2}} \left(\int_{L < |y| < 2 L} |v(y)|^2 \, dy \right)^{p/2} < L^{N-(N+1)p + \d'}.
\]
This implies that $v \in L^{p}$, for all $\frac{N}{N+1} < p < 2$. Furthermore, if $p>2$, then trivially,
\[
\int_{L < |y| < 2 L} |v(y)|^p dy \lesssim   \frac{L^{(p-2)(1-\d)}}{L^{N+2 - \d'}},
\]
which implies $v\in L^p$, if $p \leq 4+N$. In summary, $v\in L^2$ implies 
\[
v \in \bigcap_{\frac{N}{N+1} < p \leq N+4} L^p(\R^N).
\]
Let us note that under the assumption $|v(y)| \gtrsim \frac{1}{|y|^{N+1-\d}}$ these solutions have been already excluded in \cite{cs}. So, the implications above are somewhat more general. 

In the case $\a < N/2$, the natural energy bound on $v$, coming from the boundedness of the global energy, is
\[
\int_{|y| \leq L} |v|^2\, dy \lesssim L^{N-2\a},
\]
while the energy drain condition becomes
\[
\int_{|y| \leq L} |v|^2\, dy \lesssim L^{N-2\a} o(1), \text{ as } L\ra \infty.
\]
Notice that since $N-2\a >0$ this is equivalent to a similar condition over dyadic shells. Again, according to \thm{t:merger} the improved energy bound becomes
\[
\int_{|y| \leq L} |v|^2\, dy \lesssim L^{N-2 -4\a + \d'},
\]
for any $\d'>0$. This excludes solutions in the range $\a >\frac{N-2}{4}$, while in the range $0< \a \leq  \frac{N-2}{4}$ is inconsistent with the bound from below $|v(y)| \gtrsim \frac{1}{|y|^\b}$ for any $\b< 2\a+1$. We summarize our observations in the following corollary.

\begin{corollary}\label{c:ss}
Suppose $v,q$ is a self-similar solution to \eqref{ee} in the form \eqref{e:ss}. Suppose $|v | \lesssim |y|^{1-\d}$. Then $v = 0$ or nonexistent in any of the following cases
\begin{itemize}
\item $\a = \frac{N}{2}$, $v \in L^2 \backslash \bigcap_{\frac{N}{N+1} < p \leq N+4} L^p$;
\item $\frac{N-2}{4} < \a < \frac{N}{2}$, $\int_{|y| <L} |v|^2\, dy \leq L^{N-2\a} o(1)$;
\item $0< \a \leq \frac{N-2}{4}$, $\int_{|y| <L} |v|^2\, dy \leq L^{N-2\a} o(1)$, and $\frac{1}{|y|^\b}\lesssim |v(y)|$, for some $\b < 2\a +1$.
\end{itemize}
\end{corollary}

\section{Adaptations to the Navier-Stokes system} \label{s:NSE}

Most of the results of the previous sections carry over to the viscous case too. Here we assume $N = 3$. Even though some of them may not be optimal for this case, we will show that the contribution of the linear term in most cases is of lower order, and thus is subordinate to the contribution of the nonlinear term. We start with the concentration results. Let us assume that $(u,p)$ is a solution to the Navier-Stokes equation 
\begin{equation}\label{nse}
\begin{split}
u_t + u \cdot \n u + \n p & =\nu \Delta u\\
\n \cdot u & = 0\\
\end{split}
\end{equation}
with smooth initial data $u_0$. Suppose that $[0,T)$ is an interval of regularity of $u$. It was shown in \cite{cfl} that the condition $u \in L^1L^\infty $ holds automatically on $[0,T)$.
\begin{corollary}
Let $u$ be a regular solution to the Navier-Stokes equation on an interval $[0,T)$. Suppose \eqref{enlimit} holds. Then $\E_T$ has no atoms. In particular, $\E_T$ has no atoms if $u\in L^{5/3}([0,T);L^\infty(\O))$.
\end{corollary}
\begin{proof} Two additional terms that appear on the right hand side of the energy equality are
\begin{equation}\label{ }
-\nu \int_t^T \int_{\O} | \n u|^2 \s_\r \, dx d\t  + \frac{\nu}{2} \int_t^T \int_{\O} |u|^2 \D \s_\r \, dx d\t.
\end{equation}
The first term has a negative sign,  so it can be dropped. For the second term we have the estimate
\[
\begin{split}
\int_t^T \int_{\O} |u|^2 \D \s_\r \, dx d\t & \leq \frac{1}{\r^2} \int_t^T \int_{\O} |u|^2 dx d\t \leq \int_t^T \left(\int_{\O} |u|^6 dx \right)^{1/3} d\t \\
&\leq \int_t^T \int_{\O} |\n u|^2 dx  d\t \ra 0,
\end{split}
\]
as $t \ra T$.

\end{proof}

A minor modification makes it possible to extend the conclusion of \thm{t:r} to larger values of $r$.
Let us note however that if $r \geq 2$, then $u$ satisfies the Prodi-Serrin regularity condition, and hence, \thm{t:r} holds trivially.
\begin{corollary} Suppose $u \in L^{r}([0,T); L^\infty)$, $5/3< r<7/4$,  and $u$ is a regular solution to the Navier-Stokes equation on the interval $[0,T)$. Then 
\[
 d_T \geq \frac{3r - 5}{r-1}.
\]

\end{corollary}
\begin{proof}
Let us keep the original forward direction of time $t \ra T$. The initial energy inequality \eqref{ini} should be replaced with
\begin{equation*}
\begin{split}
E(t,\r) & \lesssim E(T-t_0,\r) + \frac{1}{\r} \mathrm{sum}\\
& -\nu \int_{T - t_0}^t \int_{\O} | \n u|^2 \s_\r \, dx d\t  + \frac{\nu}{2} \int_{T - t_0}^t \int_{\O} |u|^2 \D \s_\r \, dx d\t.
\end{split}
\end{equation*}
for all $T-t_0 < t<T$. The negative viscous term on the right hand side can be dropped as before, while the other viscous term can be estimated by
\[
\frac{\nu}{2} \int_{T - t_0}^t \int_{\O} |u|^2 \D \s_\r \, dx d\t \lesssim t_0/\r^2.
\]
Subsequent iterations of this term result in the sum
\[
\frac{t_0}{\r^2}( 1+ F(t)/\r + \ldots + \frac{1}{M!}(F(t)/ \r)^M ),
\]
which, given the choice of $t_0$, is comparable to $\r^{\frac{2-r}{r-1} + }$. In order for this term to be less than the required $\r^{\frac{3r-5}{r-1} - }$ the exponent $r$ has to satisfy $r<7/4$.

\end{proof}

Clearly from the proof, if $r\geq7/4$, then the dimension becomes $d < \frac{2-r}{r-1}$. However, it is somewhat unnatural that it becomes smaller, hence more singular, as $r$ approaches its regularity threshold value $r = 2$.

The local merger results of Section~\ref{s:merger} carry over backward in time. This is because the basic energy inequality \eqref{en-back} has time direction reversed. Thus, assuming that the merger time is $0$, and for $t>0$, the extra viscous term that appears on the right hand side is
\[
\frac{\nu}{2} \int_{0}^t \int_{\O} |w|^2 \D \s_\r \, dx d\t \leq \frac{C}{\r^2} \int_0^t E(\t,\r) \, d\t.
\]
This term further appears in \eqref{vanen3}, and does not effect the rest of the proof.


\begin{thebibliography}{10}


\bibitem{bp}
Olus~N. Boratav and Richard~B. Pelz.
\newblock Direct numerical simulation of transition to turbulence from a
  high-symmetry initial condition.
\newblock {\em Physics of Fluids}, 6(8):2757--2784, 1994.

\bibitem{bmvps}
M.~E. Brachet, M.~Meneguzzi, A.~Vincent, H.~Politano, and P.~L. Sulem.
\newblock Numerical evidence of smooth self-similar dynamics and possibility of
  subsequent collapse for three-dimensional ideal flows.
\newblock {\em Physics of Fluids A: Fluid Dynamics}, 4(12):2845--2854, 1992.


\bibitem{cs}
Dongho Chae, Roman Shvydkoy.
\newblock On formation of a locally self-similar collapse in the incompressible {E}uler equations.
\newblock available at {arXiv:1201.6009v2}.

\bibitem{chae-11}
Dongho Chae.
\newblock On the self-similar solutions of the 3{D} {E}uler and the related
  equations.
\newblock {\em Comm. Math. Phys.}, 305(2):333--349, 2011.

\bibitem{cfl}
Diego C{\'o}rdoba, Charles Fefferman, Rafael de la Llave.
\newblock On squirt singularities in hydrodynamics.
\newblock {\em SIAM J. Math. Anal.}, 36(1):204--213, 2004.
	



\bibitem{he-ext}
Xinyu He.
\newblock An example of finite-time singularities in the 3d {E}uler equations.
\newblock {\em J. Math. Fluid Mech.}, 9(3):398--410, 2007.

\bibitem{kerr}
Robert~M. Kerr.
\newblock Evidence for a singularity of the three-dimensional, incompressible
  euler equations.
\newblock {\em Physics of Fluids A: Fluid Dynamics}, 5(7):1725--1746, 1993.

\bibitem{kimura}
Yoshifumi Kimura.
\newblock Self-similar collapse of 2d and 3d vortex filament models.
\newblock {\em Theoretical and Computational Fluid Dynamics}, 24:389--394,
  2010.
\newblock 10.1007/s00162-009-0175-9.

\bibitem{ngb}
C.~S. Ng and A.~Bhattacharjee.
\newblock Sufficient condition for finite-time singularity and tendency towards
  self-similarity in a high-symmetry flow.
\newblock In {\em Tubes, sheets and singularities in fluid dynamics
  ({Z}akopane, 2001)}, volume~71 of {\em Fluid Mech. Appl.}, pages 317--328.
  Kluwer Acad. Publ., Dordrecht, 2002.

\bibitem{pelz}
R.~B. Pelz.
\newblock Locally self-similar, finite-time collapse in a high-symmetry vortex
  filament model.
\newblock {\em Phys. Rev. E}, 55:1617--1626, Feb 1997.

\bibitem{schonbek}
Maria Schonbek.
\newblock Nonexistence of pseudo-self-similar solutions to incompressible {E}uler equations
\newblock {\em Acta Math. Sci.}, 31B(6):1--8, 2011.


\end{thebibliography}
\end{document}